\documentclass{entcs}
\usepackage{entcsmacro}

\usepackage{QED}

\newcommand {\ra}{\tr}
\newcommand {\ras}{\tr^*}

\newcommand {\G}{\Gamma}
\newcommand {\m}{\mu}

\newcommand {\be}{\beta}
\newcommand {\al}{\alpha}
\newcommand {\sig}{\sigma}
\newcommand {\la}{\lambda}
\newcommand {\ptv}{\; ..., \;}
\newcommand {\pts}{ \; ... \; }
\newcommand {\peq}{\preceq}

\def\G{\Gamma}

\def\l{\lambda}

\def\m{\mu}

\def\f{\rightarrow}
\def\tr{\triangleright}

\def\v{\vdash}

\def\<{\langle}
\def\>{\rangle}
\def\F{\displaystyle\frac}

\def\lastname{DAVID}
\begin{document}
\begin{frontmatter}
  \title{Why the usual candidates of reducibility do not work for
the symmetric $\la \mu$-calculus}
 \author{Ren\'e DAVID}
  \address{Equipe de Logique, Universit\'e de Savoie\\
    73376 Le Bourget du Lac, France} \author{Karim NOUR}
  \address{Equipe de Logique, Universit\'e de Savoie\\
    73376 Le Bourget du Lac, France}  \thanks[myemail]{Email:
    \href{mailto:david@univ-savoie.fr} {\texttt{\normalshape
        david@univ-savoie.fr}}} \thanks[coemail]{Email:
    \href{mailto:nour@univ-savoie.fr} {\texttt{\normalshape
        nour@univ-savoie.fr}}}

\begin{abstract}
The  symmetric $\lambda \mu$-calculus is the  $\lambda
\mu$-calculus introduced by Parigot in which the reduction rule
$\m'$, which is the symmetric of $\m$, is added. We give examples
explaining why the technique using the usual candidates of
reducibility does not work. We also prove a standardization
theorem for this calculus.
\end{abstract}
\end{frontmatter}

\section{Introduction}

Since it has been understood that  the Curry-Howard isomorphism
relating proofs and programs  can be extended to classical logic,
various  systems have been introduced:  the $\l_c$-calculus
(Krivine \cite{Kri}), the $\la_{exn}$-calculus (de Groote
\cite{deG4}), the $\l \mu$-calculus (Parigot \cite{Par1}), the
$\lambda^{Sym}$-calculus (Barbanera \& Berardi \cite{BaBe}), the
$\lambda_{\Delta}$-calculus (Rehof \& Sorensen \cite{ReSo}), the
$\overline{\lambda}\mu\tilde{\mu}$-calculus (Curien \& Herbelin
\cite{cuhe}),  ...

The first calculus which respects the intrinsic symmetry of
classical logic is $\lambda^{Sym}$. It is somehow different from
the previous calculi since the main connector is not the arrow as
usual but the connectors {\em or} and {\em and}. The symmetry of
the calculus comes from the de Morgan laws.

The second calculus respecting this symmetry has been
$\overline{\lambda}\mu\tilde{\mu}$. The logical part is the
(classical) sequent calculus instead of natural deduction.

Natural deduction is not, intrinsically, symmetric but Parigot has
introduced the so called {\em Free deduction} \cite{Pa01} which is
completely symmetric. The  $\l \mu$-calculus  comes from there. To
get a confluent calculus he had, in his terminology, to fix the
inputs on the left. To keep the symmetry, it is enough to keep the
same terms and to add a new reduction rule (called the
$\mu'$-reduction) which is the symmetric rule of the
$\mu$-reduction and also corresponds  to the elimination of a
 cut. We get then a symmetric calculus that is called the
{\em symmetric $\lambda \mu$-calculus}.

The $\m'$-reduction has been considered by Parigot for the
following reasons. The $\l \mu$-calculus (with the
$\beta$-reduction and the $\mu$-reduction) has good properties :
confluence in the un-typed version, subject reduction and strong
normalization in the typed calculus. But this system has, from a
computer science point of view, a drawback: the unicity of the
representation of data is lost. It is known that, in the
$\la$-calculus, any term of type $N$ (the usual type for the
integers) is $\be$-equivalent to a Church integer. This no more
true in the $\la\m$-calculus and we can find normal terms of type
$N$ that are not Church integers. Parigot has remarked that by
adding the $\mu'$-reduction and some simplification rules the
unicity of the representation of data is recovered and subject
reduction is preserved, at least for the simply typed system, even
though the confluence is lost.

Barbanera \& Berardi proved the strong normalization of the
$\lambda^{Sym}$-calculus by using candidates of reducibility but,
unlike the usual construction (for example for Girard's system
$F$), the definition of the interpretation of a type needs a
rather complex fix-point operation. Yamagata \cite{Yam} has used
the same technique to prove the strong normalization of the
$\be\m\m'$-reduction where the types are those of system $F$ and
Parigot, again using the same ideas, has extended Barbanera \&
Berardi's result to a logic with second order quantification.

The following property trivially holds in the $\l\m$-calculus:
 \\If
$(\la x M \; N \; P_1 ... P_n) \ras (\l x M' \; N' \; P'_1 ...
P'_n) \tr (M'[x:=N'] \; P'_1 ... P'_n)$, then we may start the
reduction by reducing the $\be$ redex, i.e $(\l x M \; N \; P_1
... P_n) \tr (M[x:=N] \; P_1 ... P_n) \ras (M'[x:=N'] \; P'_1 ...
P'_n)$. This point is the key in the proof of two results for this
calculus:

\medskip

\noindent {\bf (1)} If $N$ and $(M[x:=N] \; P_1 ... P_n)$ are in
$SN$,
  then so is $(\l x M \; N \; P_1 ... P_n)$. Similarly, if $N$ and
  $(M[\al=_rN] \; P_1 ... P_n)$ are in $SN$,
  then so is $(\m \al M \; N \; P_1 ... P_n)$. They are at the
  base
  of the proof of the strong normalization of the typed
  calculus.

\noindent {\bf (2)}   The standardization theorem.

\medskip

Even though this result remains (trivially) true in the symmetric
$\l\m$-calculus and the standardization theorem still holds in
this calculus, point (1) above is no more true. This simply comes
from the fact that an infinite reduction of $(\la x M \; N)$ does
not necessarily reduce the $\be$ redex (and similarly for $(\m \al
M \; N)$) since it can also reduce the $\m'$ redex.

The other key point in the proof of the strong normalization of
typed
  calculus is the following property which remains true in the symmetric $\l\m$-calculus.

\medskip

\noindent {\bf (3)}  If $M_1, ..., M_n$ are in $SN$, then so is
$(x \; M_1 ...\;
  M_n)$.
\medskip

This paper is organized as follows. Section \ref{s2} defines the
symmetric $\l\m$-calculus and its reduction rules. We give the
proof of (3) in section \ref{xti}. Section \ref{ce} gives  the
counter-examples for (1). Finally we prove the standardization
theorem in section \ref{sta}.

\section{The symmetric $\l\m$-calculus }\label{s2}

The set (denoted as ${\cal T}$) of $\l\m$-terms or simply terms is
defined by the following grammar where $x,y,...$ are
$\lambda$-variables and $\al, \be, ...$ are $\mu$-variables:
$$
{\cal T} ::= x \mid \l x {\cal T} \mid ({\cal T} \; {\cal T}) \mid
\mu \al {\cal T} \mid   (\al \; {\cal T})
$$

Note that we adopt here a more liberal syntax (also called de
Groote's calculus) than in the original calculus since we do not
ask that a $\m \al$ is immediately followed by a $(\be \; M)$
(denoted $[\be] M$ in Parigot's notation). 

Even though this paper is only concerned with the un-typed
calculus, the $\l\m$-calculus comes from a Logic and, in
particular, the $\m$-constructor comes from a logical rule. To
help the reader un-familiar with it, we give below the typing and
the reduction  rules.

The types are those of the simply typed $\l\m$-calculus i.e. are
built from atomic formulas and the constant symbol $\perp$ with
the connector $\rightarrow$. As usual $\neg A$ is an abbreviation
for $A \rightarrow \perp$.

 The typing rules are given by figure 1 below
where $\G$ is a context, i.e. a set of declarations of the form $x
: A$ and $\al : \neg A$ where $x$ is a $\l$ (or intuitionistic)
variable, $\al$ is a $\m$ (or classical) variable and $A$ is a
formula.

\begin{center}
$\F{}{\G , x : A \v x : A} \, ax$

\medskip
$\F{\G, x : A \v M : B} {\G \v \l x M : A \f B} \, \f_i$
\hspace{0.5cm}  $\F{\G \v M : A \f B \quad \G \v N : A} {\G \v (M
\; N): B }\, \f_e$

\medskip

$\F{\G , \al : \neg A  \v M : \bot} {\G \v \mu \al M : A }  \,
\bot_e$ \hspace{0.5cm} $\F{\G , \al : \neg A  \v M : A} {\G, \al :
\neg A  \v ( \al \; M) : \bot } \, \bot_i$
\medskip

Figure 1.
\end{center}

Note that, here, we also have changed Parigot's notation but these
typing rules are those of his classical natural deduction. Instead
of writing
$$M : (A_1^{x_1 }, ..., A_n^{x_n }  \vdash   B,  C_1^{\al_1},
..., C_m^{\al_m })$$ we have written
$$x_1 : A_1, ..., x_n : A_n,
\al_1: \neg C_1, ..., \al_m : \neg C_m \vdash M : B$$

The cut-elimination procedure corresponds to the reduction rules
given below. There are three kinds of cuts.

\begin{itemize}
\item A {\em logical cut}  occurs when the introduction of the
 connective $\f$ is immediately followed by its
 elimination.  The corresponding reduction rule (denoted by $\be$) is:

$$(\l x M \; N) \tr M[x:=N]$$

\item A {\em classical cut} occurs when $\bot_e$ appears as the left
 premiss of a $\f_e$. The corresponding reduction rule (denoted by
 $\m$) is:

$$(\mu \al M \; N) \tr \mu \al M[\al=_r N]$$

where $M[\al=_r N]$ is obtained by replacing each sub-term of $M$
of the form $(\al \; U)$ by $(\al \; (U \; N))$.

\item A {\em symmetric classical cut}  occurs when $\bot_e$ appears as
 the right premiss of a $\f_e$. The corresponding reduction rule (denoted by $\m'$) is:

$$(M \; \mu \al N) \tr \mu \al N[\al=_l M]$$

 where $N[\al=_l M]$
is obtained by replacing each sub-term of $N$ of the form $(\al \;
U)$ by $(\al \; (M \; U))$.
\end{itemize}

\noindent {\bf Remark}

It is shown in \cite{Par1} that the $\be\m$-reduction is confluent
but neither $\m\m'$ nor $\be\m'$ is. For example $(\m \al x \, \m
\be y)$ reduces both to $\m \al x$ and to $\m \be y$. Similarly
$(\la z x \; \m \be y)$ reduces both to $x$ and to $\m \be y$.

\medskip

The following property is straightforward.

\begin{theorem}

If $\G \vdash M : A$ and $M \tr M'$ then $\G \vdash M' : A$.

\end{theorem}

\section{If $M_1, ..., M_n$ are in $SN$, then so is $(x \; M_1
...\; M_n)$} \label{xti}

The proofs  are only sketched. More details can be found in
\cite{tlca} where an arithmetical proof of the strong
normalization of the $\be\m\m'$-reduction for the simply typed
calculus is given.

\begin{definition}
\begin{itemize}
\item $cxty(M)$ is the number of  symbols occurring in $M$.
\item  We denote by  $N \leq M$ (resp. $N < M$) the fact that $N$ is a sub-term
  (resp. a strict sub-term) of $M$.

\item The reflexive
and transitive closure of $\tr$ is denoted by $\tr^*$.
\item If $M$ is in $SN$
  i.e. $M$ has no infinite reduction, $\eta(M)$ will denote the length
of the longest reduction starting from $M$.

  \item We denote by $N \prec M$ the fact that $N \leq M'$ for some $M'$ such that $M \ras M'$
   and either $M \tr^+ M'$ or $N < M'$. We denote by $\preceq$ the reflexive closure of $\prec$.

\end{itemize}
\end{definition}

\begin{lemma}\label{lammu}
\begin{enumerate}
\item If $(M \; N) \ras \l x P$, then $M \ras \l y M_1$ and
$M_1[y := N] \ras \l x P$.
\item If $(M \; N) \ras \m \al P$, then either ($M \ras \l y M_1$ and
$M_1[y := N] \ras \m \al P$) or ($M \ras \m \al M_1$ and $M_1[\al
=_r N] \ras P$) or ($N \ras \m \al N_1$ and $N_1[\al =_l M] \ras
P$).
\end{enumerate}
\end{lemma}
\begin{proof}
Easy.
\end{proof}

\begin{lemma}\label{l30}
Assume $M,N \in SN$ and $(M \; N) \not\in SN$. Then, either ($M
\ras \la y P$ and $P[y:=N] \not\in SN$) or
  ($M \ras \mu \al P$ and $P[\al =_r N] \not\in SN$) or
 ($N \ras \mu \al P$ and $P[\al =_l M] \not\in SN$).
\end{lemma}
\begin{proof}
By induction on $\eta(M)+\eta(N)$.
\end{proof}

\begin{lemma}\label{lambda}
The term $(x \; M_1 \pts  M_n)$ never reduces to a term of the
form $\la y M$.
\end{lemma}
\begin{proof}
By induction on $n$. Use lemma \ref{lammu}.
\end{proof}

\begin{definition}
\begin{itemize}
  \item Let $M_1, ..., M_n$ be terms and $1\leq i \leq n$.
We will  denote by $M[\al=_i (M_1 \pts M_n)]$ the term $M$ in
which every sub-term of the form $(\al \; U)$ is replaced by $(\al
\; (x \; M_1 \pts M_{i-1} \; U \; M_{i+1} \pts M_n))$ .
\item We will  denote by $\Sigma_x$ the set of simultaneous substitutions of the form
$[\al_1=_{i_1}(M^1_1 \pts M^1_n),...,\al_k=_{i_k}(M^k_1 \pts
  M^k_n)]$.
\end{itemize}

\end{definition}

\begin{lemma}\label{l21}
Assume $(x \; M_1 \pts  M_n) \ras \m\al M$. Then, there is an $i$
such that $M_i \ras \m\al P$ and $P[\al=_i (M_1 \pts M_n)] \ras
M$.
\end{lemma}

\begin{proof}
By induction on $n$. Use lemmas \ref{lammu} and \ref{lambda}.
\end{proof}

\begin{lemma}\label{l22}
Assume $M_1,...,M_n \in SN$ and $(x \; M_1 \pts M_n) \not\in SN$.
Then,  there is an $1\leq i \leq n$ such that $M_i \ras \mu \al
\;U$ and $U[\al=_i (M_1 \pts M_n)] \not\in SN$.
\end{lemma}
\begin{proof}
Let $k$ be the least such that $(x \; M_1 \pts M_{k-1}) \in SN$
and $(x \; M_1 \pts M_k)$  $ \not\in SN$. Use lemmas \ref{l30},
\ref{lambda} and \ref{l21}.
\end{proof}

\begin{lemma}\label{l7a}
 Let $M$ be a term and $\sigma \in \Sigma_x$. If $M[\sigma] \ras \m\al
 P$ (resp. $M[\sigma] \ras \l x P$) , then $M\ras \m\al Q$
 (resp. $M\ras \l x Q$) for some $Q$ such that $Q[\sigma] \ras P$.
\end{lemma}

\begin{proof}
By induction on $M$.
\end{proof}

The next lemma is the key of the proof of theorem \ref{xt1}.
Though intuitively clear (if the cause of non $SN$ is the
substitution $\delta=_i(P_1 ... P_n)$, this must come from some
$(\delta \; M') \prec M$) its proof is rather technical.

\begin{lemma}\label{crux''}

Let $M$ be a term and $\sigma \in \Sigma_x $. Assume $\delta$ is
free in $M$ but not free in $Im(\sig)$. If $M[\sigma] \in SN$ but
$M[\sigma][\delta=_i(P_1 ... P_n)] \not\in
  SN$, there is $M'\prec M$ and $\sigma'$ such that
 $M'[\sigma'] \in SN$ and $(x \;
 P_1 ... P_{i-1} \; M'[\sigma']\; P_{i+1} ...P_n)\not\in
  SN$.
\end{lemma}

\begin{proof}
See \cite{tlca} for more detail.
\end{proof}

\begin{theorem}\label{xt1}
Assume $M_1,..., M_n$ are in $SN$. Then $(x \; M_1 \pts M_n) \in
SN$.
\end{theorem}
\begin{proof}
We prove a more general result. Let  $M_1,..., M_n$ be terms and
$\sig_1,..., \sig_n$ be in $\Sigma_x$.  If $M_1[\sig_1], \ptv
M_n[\sig_n] \in SN$, then $(x \; M_1[\sig_1] \pts M_n[\sig_n]) \in
SN$. This is done by induction on $(\Sigma  \eta(M_i), \Sigma
cxty(M_i))$. Assume $(x \; M_1[\sig_1]$ ... $ M_n[\sig_n])\not \in
SN$. By lemma \ref{l22}, there is an $i$ such that $M_i[\sig_i]
\ras \mu \al \;U$ and $U[\al=_i (M_1 [\sig_1] \pts M_n[\sig_n])]
\not\in SN$. By lemma \ref{l7a}, $M_i\ras \m\al Q$
 for some $Q$ such that $Q[\sigma_i] \ras U$.
  Thus $Q[\sigma_i][\al=_i (M_1 [\sig_1] \pts M_n[\sig_n])] \not\in
 SN$.
By lemma  \ref{crux''}, let  $M'\prec Q \peq M_i$ and $\sigma'$ be
such that
 $M'[\sigma'] \in SN$ and $(x \;
 M_1 [\sig_1] ... M_{i-1} [\sig_{i-1}]$   $M'[\sigma']\; M_{i+1} [\sig_{i+1}] ...M_n[\sig_n])\not\in
  SN$. This contradicts the induction hypothesis since $(\eta(M'),
cxty(M')) < (\eta(M_i), cxty(M_i))$.
\end{proof}

\section{The counter-examples} \label{ce}

\begin{definition}
 Let $U$ and $V$ be terms.
\begin{itemize}
  \item $U \hookrightarrow V$ means that each reduction of $U$ which is long enough must go
through $V$, i.e. there is  some $ n_0 $ such that, for all $n
>n_0$, if  $U =U_0 \tr U_1 \tr ... \tr U_n$ then $ U_p =
V$ for some $p$.
  \item $U \curvearrowright V$ means that $U$ has only one redex and
$U \ra V$.
\end{itemize}

\end{definition}

\noindent {\bf Remark}

It is easy to check that if $U \hookrightarrow V$ (resp. $U
\curvearrowright V$) and $V \in SN$, then $U \in SN$.

\begin{definition}
\begin{itemize}

\item Let $M_0=\la x (x \; P \; \underline{0})$ and $M_1=\la x (x \;
P \;
 \underline{1})$ where $\underline{0} = \la x \la y y$, $\underline{1} = \la x \la
 y x$, $P=\la x \la y \la z \;(y \; (z \;\underline{1} \; \underline{0}) \; (z \;
 \underline{0} \; \underline{1}) \;
\la d \underline{1}   \; \Delta  \;  \Delta)$ and $\Delta = \la x
(x \; x)$.

\item Let $M = \langle (x \; M_1) , (x \; M_0) \rangle$, $M' =
\langle (\be \; \la x (x \; M_1)), (\be \; \la x (x \; M_0))
\rangle$ where $\langle T_1,T_0 \rangle$ denotes the pair of
terms, i.e. the term $\la f (f \; T_1 \; T_0)$ where $f$ is a
fresh variable.
\item Let $N= (\al \; \la z(\al \;z))$.

\end{itemize}
\end{definition}

\begin{lemma}
\begin{enumerate}
 \item $(M_1 \; M_0), (M_0 \; M_1) \not\in SN$.
 \item $(M_0 \; M_0), (M_1 \; M_1) \in SN$.
\end{enumerate}
\end{lemma}

\begin{proof}

\begin{enumerate}
\item Assume  $i \neq j$, then
\begin{eqnarray}
(M_i \; M_j) &\tr^* & (P \; P \; \underline{j} \; \underline{i})
\nonumber \\
\, &\tr^*  &(\underline{j} \; (\underline{i} \;\underline{1} \;
\underline{0}) \; (\underline{i} \; \underline{0} \; \underline{1}) \;
\la d \underline{1}   \; \Delta  \;  \Delta) \nonumber \\
\, &\tr^*  &(\underline{0} \; \la d \underline{1}   \; \Delta  \;
\Delta)\nonumber \\
\, &\tr^*  &(\Delta \; \Delta)\nonumber
\end{eqnarray}
and thus $(M_i \; M_j) \not\in SN$.
\item It is easy to check that $(M_i \; M_i) \hookrightarrow (\underline{1}
\; \la d \underline{1}   \; \Delta  \; \Delta) \curvearrowright
(\la y \la d \underline{1} \; \Delta  \; \Delta) \curvearrowright
(\la d \underline{1} \; \Delta) \curvearrowright \underline{1}$.
\qed
\end{enumerate}
\end{proof}

\begin{proposition}
$M[x:=\m\al N] \in SN$ but $(\la x M \; \m\al N) \not\in SN$.
\end{proposition}

\begin{proof}
\noindent (a) Since $M[x:=\mu \al N] = \langle (\mu \al N \; M_1)
, (\mu \al N \; M_0) \rangle$, by theorem \ref{xt1}, to show that
$M[x:=\m\al N] \in SN$, it is enough to show that $(\mu \al N \;
M_i) \in SN$.

\begin{eqnarray}
(\mu \al N \; M_i) &\curvearrowright  & \mu \al (\al ( \la z (\al
\; (z \; M_i))
M_i)) \nonumber\\
\, &\curvearrowright  & \mu \al (\al \; (\al \; (M_i \; M_i)))
\nonumber\\
\, &\hookrightarrow  & \mu \al (\al \; (\al \;
\underline{1}))\nonumber
\end{eqnarray}

\noindent (b)

\begin{eqnarray}
 \; (\la x M \; \m\al N) &\tr^* & \m \al (\al \; (\la x M \; \la z(\al \; (\la x M \; z))))
 \nonumber \\
\, &\tr^*  &   \m \al (\al \; (\la x M \; \la z(\al \; \langle (z
\;
M_1) , (z \; M_0) \rangle )))\nonumber\\
\, &\tr^*  &  \m \al (\al \; \langle (\al \; \langle (M_1 \; M_1),
(M_1 \; M_0)\rangle) , (\al \; \langle (M_0 \; M_1) , (M_0 \;
M_0)\rangle)
\rangle)\nonumber\\
\, &\tr^*  &  \m \al (\al \; \langle (\al \; \langle
\underline{1}, (\Delta \; \Delta)\rangle) , (\al \; \langle
\underline{1} , (\Delta \; \Delta)\rangle) \rangle)\nonumber
\end{eqnarray}
and thus $(\la x M \; \m\al N) \not\in SN$.

\end{proof}

\begin{proposition}

$M'[\be=_r\m\al N] \in SN$ but $(\m \be M' \; \m\al N) \not\in
  SN$.

\end{proposition}

\begin{proof}
\noindent (a) $(\la x (x \; M_i) \; \mu \al N)$ has two redexes
thus \\ either
\begin{eqnarray}
(\la x (x \; M_i) \; \mu \al N) &\tr &(\mu \al N \; M_i) \nonumber \\
\, &\curvearrowright  &  \mu \al (\al ( \la z (\al \; (z \; M_i)) \; M_i))   \nonumber\\
\, &\curvearrowright  & \mu \al (\al \; (\al \; (M_i \; M_i)))  \nonumber\\
\, &\hookrightarrow   &\mu \al (\al \; (\al \;\underline{1}))
\nonumber
\end{eqnarray}
or
\begin{eqnarray}
(\la x (x \; M_i) \; \mu \al N) &\tr & \mu \al (\al (\la x (x \;
M_i) \;
\la z (\al \; (\la x (x \; M_i) \; z))))   \nonumber\\
\, & \hookrightarrow  & \mu \al (\al \; (\al \; (M_i \; M_i)))  \nonumber\\
\, & \hookrightarrow   &\mu \al (\al \; (\al \; \underline{1}))
\nonumber
\end{eqnarray}
Thus $(\la x (x \; M_i) \; \mu \al N) \hookrightarrow \mu \al (\al
\; (\al \; \underline{1}))$ and, by theorem \ref{xt1}, it follows
that $M'[x:=\mu \al N] = \langle (\be \; (\la x (x \; M_1) \; \mu
\al N)) , (\be \; (\la x (x \; M_0) \; \mu \al N))\rangle \in SN$.

\medskip
\noindent (b)

\begin{eqnarray}
 \; (\m \be M' \; \m\al N) &\tr^*  &\mu \al (\al \; (\m \be
M'\; \la z
(\al \; (\m \be M' \; z)) ))  \nonumber \\
\, &\tr^*  &\mu \al (\al \; (\m \be M'\; \la z (\al \; \m \be
\langle (\be \; (z \; M_1)), (\be \; (z \;
M_0))\rangle) ))  \nonumber \\
\, &\tr^*  &\mu \al (\al \; \mu \be \langle (\be \; (\al \; \mu
\be
\langle(\be \; \underline{1}) , (\be \; (\Delta \; \Delta))\rangle)),\nonumber\\
\, & \, & \;\;\;\;\;\;\;\;\;\;\;\;\;\;\; (\be \; (\al \; \mu \be
\langle(\be \; (\Delta \; \Delta)) , (\be \;
\underline{1})\rangle))\rangle)\nonumber
\end{eqnarray}
and thus $(\m \be M' \; \m\al N) \not\in SN$.

\end{proof}

\section{Standardization}\label{sta}

In this section we give a standardization theorem for the
$\be\m\m'$-reduction. It also holds  for the $\m\m'$-reduction and
its proof simply is a restriction of the other one.

\begin{definition}\label{def1}
\begin{enumerate}
  \item The sequence $(M_i)_{1\leq i \leq n}$ is standard iff one of the
  following cases hold:
\begin{enumerate}
  \item For all $i$, $M_i = \la x N_i$ (resp. $M_i=\mu \al N_i$, $M_i=(x \; N_i)$, $M_i=(\al \; N_i)$) and the
  sequence $(N_i)_{1\leq i \leq n}$ is standard
  \item There are standard sequences $(N_i)_{1\leq i \leq k}$ and
   $(P_i)_{k\leq i \leq n}$ such that, for $1\leq i \leq k$,
   $M_i=(N_i \; P_k)$ and, for $k\leq i \leq n$, $M_i=(N_k \;
   P_i)$.
   \item There is a standard sequence $(N_i)_{1\leq i \leq k}$ and
   $Q$
   such that,
\begin{enumerate}
  \item either, for $1\leq i \leq k$,
   $M_i=(N_i \; Q)$ and $N_k = \la x P$ and $N_{k-1}$ does not begin with
  $\la$ and $M_{k+1}=P[x:=Q]$ and the sequence  $(M_i)_{k+1\leq i \leq
  n}$ is standard.
  \item or, for $1\leq i \leq k$,  $M_i=(N_i \; Q)$ and $N_k = \mu \al P$
  and $N_{k-1}$ does not begin with
  $\mu$ and $M_{k+1}=P[\al=_rQ]$ and the sequence  $(M_i)_{k+1\leq i \leq
  n}$ is standard.
  \item or, for $1\leq i \leq k$,  $M_i=(Q \; N_i)$ and $N_k = \mu \be P$ and
  $N_{k-1}$ does not begin with
  $\mu$ and $M_{k+1}=P[\be=_lQ]$ and the sequence  $(M_i)_{k+1\leq i \leq
  n}$ is standard.
\end{enumerate}

\end{enumerate}
  \item $M \tr_{st} M'$ iff there is a standard sequence $(M_i)_{1 \leq
  i \leq n}$ such that $M=M_1$ and $M' = M_n$.

\end{enumerate}

\end{definition}

\noindent {\bf Remarks and notation}

\begin{itemize}
  \item The clauses in 1 above correspond to a definition by induction on
the ordered pair $(n, cxty(M_1))$.
\item It is easy to check that, restricted to the $\la$-calculus,
this definition is equivalent to the usual definition of a
standard reduction.

  \item Clearly, if $M \tr_{st} M'$ then $M \ras M'$. In this case, we will
  denote the length of the reduction by $lg(M \tr_{st} M')$.

\end{itemize}

\begin{lemma}\label{l2}
Assume $M\ra _{st}P$ and $N\ra _{st}Q$. Then : (a) $\mu \alpha
M\ra _{st}\mu \alpha P$,  (b) $\la x M \tr_{st} \la x P$, (c)
$(M\;N)\ra _{st}(P\;Q)$, (d)  $M[x:=N] \tr_{st} P[x:=Q]$  and (e)
for $j \in \{l,r\}$, $M[\alpha =_jN]\ra _{st}P[\alpha =_jQ]$.
\end{lemma}
\begin{proof}
(a), (b) and (c) are immediate. (d) and (e) are proved  by
induction on $(lg(M\ra _{st}P), cxty(M))$ and a straightforward
case analysis on the definition of a standard sequence bringing
from $M$ to $P$.
\end{proof}

\begin{lemma}\label{st1}
Assume $M\tr_{st} P$ and $P\ra Q$. Then $M\ra _{st}Q$.
\end{lemma}
\begin{proof}
This is proved  by induction on $(lg(M \tr_{st} P), cxty(M))$ and
by case analysis on the reduction $M \tr_{st} P$. The only  case
which is not immediate is the following: $M=(M_1 \; M_2) \ras (N_1
\; M_2) \ras (N_1 \; N_2)=P$ where $M_1 \tr_{st} N_1$ and $M_2
\tr_{st} N_2$. If the redex reduced in $P \ra Q$ is in $N_1$ or
$N_2$ the result follows immediately from the induction
hypothesis. Otherwise, assume, for example that $N_1=\m\al R$ and
$Q=R[\al=_r N_2]$. Let the reduction $M_1 \tr_{st} N_1$ be as
follows: $M_1 \tr_{st} \m\al R_1 \tr_{st} \m\al R$ where $\m\al
R_1$ is the first term in the reduction that begins with $\m$. It
follows then from lemma \ref{l2} that the following reduction is
standard. $M=(M_1 \; M_2) \tr_{st} (\m\al R_1 \; M_2) \ra \m\al
R_1[\al=_r M_2] \tr_{st} \m\al R[\al=_r N_2]$.
\end{proof}

\begin{theorem}\label{st2}
Assume $M\ras P$.\ Then $M\ra _{st}P$.
\end{theorem}
\begin{proof}
By induction on the length of the reduction $M \ras M_1$. The result
follows immediately from lemma \ref{st1}.
\end{proof}

\end{document}